\documentclass[12pt]{amsart}
\input xy \xyoption{all}
\usepackage{dsfont}
\usepackage{enumerate}
\usepackage{amssymb}
\usepackage{amsmath}
\usepackage[all]{xy}
\usepackage{tikz} 
\usepackage{lmodern,bm}
\usepackage{stmaryrd}
\usepackage{graphicx}

\def\acts{\curvearrowright}
\def\Ne{{\mathcal N}}
\def\Q{\mathbb Q}
\def\T{{\mathbb T}}
\def\P{\mathbb P}
\def\hbar{h}
\def\Oc{\mathcal O}

\DeclareMathOperator\RRep{Rep}
\DeclareMathOperator\R{\mathbb R}
\DeclareMathOperator\C{\mathbb C}
\DeclareMathOperator\Z{\mathbb Z}
\DeclareMathOperator\N{\mathbb N}
\DeclareMathOperator\I{\mathcal I}
\DeclareMathOperator\Sym{Sym}
\DeclareMathOperator\csm{c^{sm}}

\DeclareMathOperator\gr{Gr}

\DeclareMathOperator\Hom{\mathrm Hom}

\DeclareMathOperator\codim{codim}

\DeclareMathOperator\pt{pt}

\DeclareMathOperator{\GL}{GL}

\DeclareMathOperator\Fl{Fl}
\DeclareMathOperator\F{\mathcal F}
\DeclareMathOperator\mC{mC}
\DeclareMathOperator\mS{mS}

\newtheorem{fact}{Fact}[section]

\newtheorem{theorem}[fact]{Theorem}
\newtheorem{definition}[fact]{Definition}
\newtheorem{example}[fact]{Example}
\newtheorem{rremark}[fact]{Remark}
\newenvironment{remark}{\begin{rremark} \rm}{\end{rremark}}
\newtheorem{proposition}[fact]{Proposition}
\newtheorem{corollary}[fact]{Corollary}
\newtheorem{conjecture}[fact]{Conjecture}
\newtheorem{property}[fact]{Property}

\def\charr{characteristic }
\def\tilde#1{\widetilde{#1}}

\usepackage{pgf}
\def\obrazek#1#2{\pgfdeclareimage[height=#2]{#1}{#1}
\pgfuseimage{#1}}

\author{L\'aszl\'o M. Feh\'er}
\address{Institute of Mathematics, E\"otv\"os University Budapest, Hungary
\newline
 \indent        Alfr\'ed R\'enyi Institute of Mathematics, Hungarian Academy of Sciences}
\email{lfeher63@gmail.com}

\author{Rich\'ard Rim\'anyi}
\address{Department of Mathematics, University of North Carolina at Chapel Hill, USA}
\email{rimanyi@email.unc.edu}

\author{Andrzej Weber}
\address{Institute of Mathematics, University of Warsaw, Poland}
\email{aweber@mimuw.edu.pl}

\title{Characteristic classes of orbit stratifications, the axiomatic approach}

\begin{document}

\begin{abstract}
Consider a complex algebraic group $G$ acting on a smooth variety $M$ with finitely many orbits, and let $\Omega$ be an orbit.
The following three invariants of $\Omega\subset M$ can be characterized axiomatically: (1) the equivariant fundamental class $[\overline{\Omega}, M]\in H^*_G(M)$, (2) the equivariant Chern-Schwartz-MacPherson class $\csm(\Omega, M)\in H^*_G(M)$, and (3) the equivariant motivic Chern  class $\mC(\Omega, M) \in K_G(M)[y]$.
The axioms for Chern-Schwartz-MacPherson and motivic Chern classes are motivated by the axioms for cohomological and K-theoretic stable envelopes of Okounkov and his coauthors.
For $M$ a flag variety and $\Omega$ a Schubert cell---an orbit of the Borel group acting---this implies that CSM and MC classes coincide with the weight functions studied by
Rim\'anyi-Tarasov-Varchenko.
In this paper we review the general theory and illustrate it with examples.
\end{abstract}
\maketitle

\section{Introduction}

An effective way of studying the subvariety $X$ of a smooth complex variety $M$ is assigning a \charr class to $X\subset M$, living in the cohomology or K-theory of the ambient space $M$.
When $M$ is a $G$-representation and $X$ is invariant, then the \charr class can be considered in the richer $G$-equivariant cohomology or K-theory of $M$.

In this paper we consider three flavors of \charr classes: the fundamental class (in cohomology), the Chern-Schwartz-MacPherson (CSM) class (in cohomology) and the motivic Chern class (MC class) in K-theory. These \charr classes encode fine geometric and enumerative properties of $X\subset M$, and their theory overlaps with the recent advances relating geometry with quantum integrable systems.

We will be concerned with the following situation: a complex algebraic group $G$ acts on a smooth variety $M$ with finitely many orbits, and we want to find the \charr classes listed above associated to the orbits or their closures. This situation is frequent in Schubert calculus and other branches of enumerative geometry. By the nature how \charr classes are defined, such calculation assumes resolutions of the singularities of the orbit closures. The main message of the theorems we review in Sections \ref{sec-fund}, \ref{sec:CSM}, \ref{sec:MC}, is that such a resolution calculation can be replaced with an axiomatic approach for all three
types of classes mentioned above. For the fundamental class this fact has been known for two decades \cite{Rimanyi, FRbrasil}, but for the other two mentioned
classes this is a very recent development stemming from the notion of Okounkov's stable envelopes ~\cite{OkounkovK}.

As an application of the axiomatic approach the \charr classes of ordinary and matrix Schubert cells can be computed (in type A, for arbitrary partial flag manifolds)---we present the formulas in Section \ref{sec:Fl}. For the motivic Chern class one of the axioms concerns convex geometry, namely the containment of one convex polytope in another one. Throughout our exposition  we try to illustrate this fascinating connection between complex geometry and the theory of convex polytopes with 2- and 3-dimensional pictures.

It is a classical fact that \charr classes of geometrically relevant varieties often exhibit positivity (or alternating sign) properties.
Having computed lots of examples of equivariant motivic Chern classes, we observe these properties for the motivic Chern class as well. In Section~\ref{sec:pos} we collect three
different positivity conjectures on MC classes of Schubert cells. One of these conjectures is a K-theoretic counterpart of properties of CSM classes studied by Aluffi-Mihalcea-Sch\"urmann-Su \cite{AM, AMSS}.  Another one of our conjectures is a new phenomenon for \charr classes: not the sign of some coefficients are conjectured, but rather that the coefficient polynomials are log-concave, cf. \cite{Huh}.

We finish the paper with studying the \charr classes of varieties in the source and in the target of a GIT map. The fundamental class of an invariant  subvariety and the fundamental class of its image under the GIT map are essentially the same (more precisely, one lives in a quotient ring, and the other one is a representative of it). A familiar fact illustrating this phenomenon is that the fundamental class of both the matrix Schubert varieties and the ordinary Schubert varieties are the Schur functions. However, for the CSM and the MC classes the situation is different: these classes before and after the GIT map are not the same (or rather, one does not represent the other). Yet,  there is direct relation between them, which we prove in Section \ref{sec-laci}.

\smallskip

The goal of most of this paper is to review and illustrate concepts, although in Section \ref{sec-alt} we complement a proof of \cite{FRW} with an alternative argument. The mathematical novelties are the conjectures in Section \ref{sec:pos}, the careful treatment of the pull-back property of motivic Segre classes in Section \ref{sec:trans}, and its consequence, Theorem \ref{te:eq-git} which relates motivic Segre classes of varieties in the source and in the target of a GIT quotient.

\smallskip

\noindent{\bf Acknowledgments.}
L.F. is supported by NKFI grants K 112735 and KKP 126683. R.R. is supported by the Simon Foundation grant 523882. A.W. is supported by NCN grant 2013/08/A/\-ST1/\-00804.

\section{Fundamental class in equivariant cohomology}\label{sec-fund}

By equivariant cohomology we will mean Borel's version, developed in \cite{Bo}. Our general reference is \cite{Bt}.


Suppose a complex algebraic group $G$ acts on the smooth complex algebraic variety $M$. Suppose also that the action has finitely many orbits, and let $\Oc$ be the set of orbits. For an orbit $\Omega \in \Oc$ let $G_\Omega$ denote the stabilizer subgroup of a point in $\Omega$. Then additively we have
$$H^*_G(M;\Q)\simeq \bigoplus_{\Omega\in\Oc} H_G^{*-2\codim(\Omega)}(\Omega;\Q).$$
The decomposition holds because each cohomology group $H_G^{* }(\Omega;\Q)\simeq H^{*}(BG_\Omega;\Q)$ is concentrated in even degrees.  Suppose further that the normal bundle $\nu_\Omega$  of each orbit $\Omega$ has nonzero equivariant Euler class
\begin{equation}\label{Atiyahcondition}
e(\nu_\Omega)\neq 0\in H^*_G(\Omega;\Q).
\end{equation}
Then a class in $H^*_G(M;\Q)$ is determined by the restrictions to orbits, cf. \cite[\S9]{AtBo}. The topic of this paper is that certain characteristic classes associated with orbits are determined by less information: they are determined by some {\em properties} of their restrictions to the orbits.

\begin{example}\rm For $M=\C$, $G=\C^*$ with the natural action on $\C$.  We have the short exact sequence
$$\begin{matrix}0&\longrightarrow &H^{*-2}_{\C^*}(\{0\},\Q)&\longrightarrow&H^{*}_{\C^*}(\C;\Q)&\longrightarrow &H^{*}_{\C^*}(\C^*;\Q)&\longrightarrow &0.\\
&&\parallel&^{t\cdot}&\parallel&^{ev_{t=0}}&\parallel\\
&&\Q[\,t\,]&&\Q[\,t\,]& &\Q\end{matrix}$$
In this example the restriction to $\{0\}$ is obviously an isomorphism. In general, when $M$ is an equivariantly formal space (see \cite{GKM}) and $G$ is a torus, then the restriction to the fixed point set is a monomorphism.
\end{example}

The most natural \charr class associated to an orbit $\Omega$ is the fundamental class of its closure, which we will denote by $[\overline{\Omega}]=[\overline{\Omega}, M]\in H_G^*(M,\Q)$. In an appropriate sense, it is the Poincar\'e dual of the homology fundamental class of $\overline{\Omega}$. Such classes are studied in Schubert calculus under the name of Schubert classes, in singularity theory under the name of Thom polynomials, and in the theory of quivers under the name of  quiver polynomials.
Here is an axiomatic characterization of the equivariant fundamental class, which leads to an effective method to calculate them.

\begin{theorem} [\cite{Rimanyi, FRbrasil}] \label{Thom}
Suppose $M$ has finitely many $G$-orbits and \eqref{Atiyahcondition} holds. Then the equivariant fundamental classes of orbit closures are determined by the conditions:
\begin{enumerate}[(i)]
\item (support condition) the class $[\overline{\Omega}]$ is supported on $\overline{\Omega}$;
\item (normalization condition) $[\overline{\Omega}]_{|\Omega}=e(\nu_\Omega)$;
\item (degree condition) $\deg([\overline{\Omega}])=2\codim(\Omega)$.
\end{enumerate}
We obtain an equivalent system of conditions if we replace condition (i) with any of the following two:
\begin{enumerate}[(i)]
\setcounter{enumi}{3}
\item (homogeneous equations, ver1) $[\overline{\Omega}]_{|\Theta}=0$, if $\Theta\not\subset \overline{\Omega}$,
\item (homogeneous equations, ver2) $[\overline{\Omega}]_{|\Theta}=0$, if $\codim(\Theta)\leq \codim(\Omega)$, $\Theta\not=\Omega$.
\end{enumerate}
We also obtain an equivalent system of conditions if we replace condition (iii) with
\begin{enumerate}[(i)]
\setcounter{enumi}{5}
\item (modified degree condition) $\deg([\Omega])_{|\Theta}<\deg([\Theta])_{|\Theta}$ for $\Theta\neq \Omega$.
\end{enumerate}
\qed \end{theorem}

The advantage of {\em (iv)} or {\em (v)} over {\em (i)} is that they are local conditions, and hence explicitly computable. The modified degree condition is only added to illustrate the similarity with analogous  characterizations of other \charr classes in Sections \ref{sec:CSM}, \ref{sec:MC}.

\begin{example} \rm \label{exabove}
Let $G=\GL_2(\C)\times \GL_3(\C)$ act on the vector space $\Hom(\C^2,\C^3)$ via $(A,B)\cdot X=BXA^{-1}$. The orbits of this action are $\Omega^i=\{X: \dim(\ker(X))=i\}$ for $i=0,1,2$. 
 The fundamental class $[\overline{\Omega^1}]$ is a homogeneous degree 2 polynomial in $\Z[a_1,a_2,b_1,b_2,b_3]^{S_2\times S_3}$ (where $\deg a_i=\deg b_i=1$). The constraints put on this polynomial by the conditions of Theorem \ref{Thom} are $\phi_0([\overline{\Omega^1}])=0$ (condition (v)) and $\phi_1([\overline{\Omega^1}])=(b_2-a_2)(b_3-a_2)$ (condition (ii)), where $\phi_0$ is the ``restriction to $\Omega^0$'' map, namely $a_1\mapsto t_1, a_2\mapsto t_2,b_1\mapsto t_1, b_2\mapsto t_2,b_3\mapsto b_3$, and $\phi_1$ is the ``restriction to $\Omega^1$'' map, namely $a_1\mapsto t_1, a_2\mapsto a_2,b_1\mapsto t_1, b_2\mapsto b_2,b_3\mapsto b_3$. Calculation shows that the only solution to these two constraints is $[\overline{\Omega^1}]=A_1^2-A_2+B_2-A_1B_1$, where $A_i$'s ($B_i$'s) are the elementary symmetric polynomials of the $a_i$'s ($b_i$'s).

More details on this calculation as well as on analogous calculations for `similar' representations (e.g. Dynkin quivers, $\Lambda^2\C^n$, $S^2\C^n$) can be found e.g. in \cite{FRss, FRbrasil, FRduke}. However, these fundamental classes can also be computed by other methods (resolutions, degenerations). The real power of Theorem \ref{Thom} is that it is applicable even in situations where deeper geometric information such as resolutions or degenerations are not known---eg. contact singularities, matroid representations spaces, see \cite{Rimanyi, matroid}.
 \end{example}

\section{Chern-Schwartz-MacPherson classes of orbits}\label{sec:CSM}

Another cohomological \charr class associated with an invariant subset $X$ of $ M$ is the Chern-Schwartz-MacPherson (CSM) class $\csm(X, M)\in H_G^*(M;\Q)$, see e.g. \cite{MacP, Ohmoto, WeCSM, AM} for foundational literature and \cite{FRcsm} for the version we consider. The CSM class is an inhomogeneous class, whose lowest degree component equals the fundamental class $[\overline{X}, M]$.

In general CSM classes are defined for (equivariant) constructible functions on $M$. Here we will only consider CSM classes of locally closed invariant smooth subvarieties of $M$ (corresponding to the indicator functions of such varieties).

According to Aluffi \cite{ACCBV} the
CSM classes can be calculated along the following lines: one finds $Y$ a partial completion of $X$ which maps properly to $M$ and such that $Y\setminus X$ is a smooth divisor with normal crossings $D=\bigcup_{i=1}^m D_i$. That is, we have a diagram
\[
\xymatrix{{X'}\ar[d]_{\simeq}\ar@{^{(}->}[r]& Y \ar[d]^{f}\\
{X}\ar@{^{(}->}[r]& M
}
\]
where $f$ is a proper map. Then
\begin{equation}\label{csmdef}\csm(X , M)=f_*\left(c(T_Y)-\sum_i c(T_{D_i})+\sum_{i<j} c(T_{D_i\cap D_j})-\dots \right).\end{equation}

\medskip
In our situation, this method demanding knowledge about a resolution, can be replaced with an axiomatic characterization.

For an orbit $\Omega\subset M$ let $x_\Omega\in \Omega$, and let $G_\Omega$ be the stabilizer subgroup of $x_\Omega$. Denote $T_\Omega=T_{x_\Omega}\Omega$, and $\nu_\Omega=T_{x_\Omega}M-T_\Omega$ as $G_\Omega$-representations. By the degree of an inhomogeneous cohomology class $a=a_0+a_1+\dots+a_d$ with $a_i\in H^i_G(\Omega)$ and $a_d\neq0$ we mean $d$.

\begin{theorem} [\cite{RV, FRcsm}] \label{csminter}
Suppose $M$ has finitely many $G$-orbits and \eqref{Atiyahcondition} holds. Then the equivariant CSM classes of the orbits are determined by the conditions:
\begin{enumerate}[(i')]
\item (divisibility condition) for any $\Theta$ the restriction $\csm(\Omega, M)_{|\Theta}$ is divisible by $c(T_\Theta)$ in $H^*_G(\Theta;\Q)$;
\item (normalization condition) $\csm(\Omega,  M)_{|\Omega}=e(\nu_\Omega)c(T_\Omega)$;
\item (smallness condition) $\deg(\csm(\Omega, M)_{|\Theta})<\deg(\csm(\Theta , M)_{|\Theta})$ for $\Theta\not=\Omega$.
\end{enumerate}
\qed
\end{theorem}

\begin{example}\rm \label{exabove2}
Continuing Example \ref{exabove}, let us calculate $\csm(\Omega^1)$. It is an inhomogeneous polynomial in $\Z[a_1,a_2,b_1,b_2,b_3]^{S_2\times S_3}$. The constraints put on this polynomial by the conditions of Theorem \ref{csminter} are 
$\phi_0(\csm(\Omega^1))=0$,
$\phi_1(\csm(\Omega^1))=(b_2-a_2)(b_3-a_2)(1+b_2-t_1)(1+b_3-t_1)(1+t_1-a_2)$, and 
$\deg(\phi_2(\csm(\Omega^1)))<6$. Here $\phi_1$, $\phi_2$ are as in Example \ref{exabove}, and $\phi_3$ is the ``restriction to $\Omega^2$'' map which turns out to be the identity map. Calculation shows that the unique solution to these constraints is (see notations in Example \ref{exabove})
\begin{multline*}
\csm(\Omega^1)=
(A_1-A_2+B_2-A_1B_1)+(-A_1^3+2A_1^2B_1-A_1B_1^2-A_1B_2+B_1B_1-B_3)+\\ \ldots+(-3A_1A_2^2+\ldots+2B_2B_3).
\end{multline*}
Details on such calculations, other examples of interpolation calculations of CSM classes, as well as more conceptual ways of presenting the long CSM polynomials are in \cite{FRcsm, Job}.

\end{example}

\section{Characteristic classes in equivariant K-theory}\label{sec:MC}

\subsection{The motivic Chern class}
We consider the topological equivariant K-theory constructed by Segal \cite{Segal} or Thomason's algebraic  K-theory, see \cite[\S5]{ChGi}. It is compatible with the K-theory built from locally free sheaves for algebraic varieties.
As before we have a decomposition
$$K_G(M)\simeq \bigoplus_{\Omega\in\Oc} K_G(\Omega)\,$$
and $K_G(\Omega)$ is isomorphic with the representation ring $R(G_\Omega)$.

The analogue of the Euler class of a vector bundle $E$ in K-theory is provided by the $\lambda$-operation
$$e^K(E)=\lambda_{-1}(E^*)=1-E^*+\Lambda^2E^*-\Lambda^3E^*+\dots.$$
It has the property that for any submanifold $\iota:N\hookrightarrow M$ and an element $\beta\in K_G(N)$ over $N$ we have $$\iota^*\iota_*(\beta)=e^K(\nu_N)\cdot \beta\,.$$
The total Chern class in K-theory is defined by
$$c_G^K(E)=\lambda_y(E^*)=1+yE^*+y^2\Lambda^2E^*+y^3\Lambda^3E^*+\dots\in K_G(M)[y]\,.$$
Note that, both in cohomology and K-theory, the total Chern class can be interpreted as the Euler class of the bundle $E$ which is equivariant with respect to $G\times \C^*$, where $\C^*$ acts on the base trivially, but acts on the bundle by scalar multiplication. Identifying
$ K_{G\times \C^*}(M)$ with $ K_G(M)[h,h^{-1}]$ and setting $y=-h^{-1}$ we obtain
$$c^K_G(E)=e^K_{G\times \C^*}(E)\,.$$
{We will drop the subscript $G$ in the notation, as we do for cohomological Chern classes.}

Our next \charr class defined for a locally closed subvariety $X$ in the smooth ambient space $M$ is the motivic Chern class (MC class) $\mC(X, M)\in K_G(M)[y]$, see \cite{BSY}. In fact this class is defined more generally for maps $X\to M$ but we will not need this generality here. The $G$-equivariant MC class
\[
\mC(X, M)\in K_G(M)[y]
\]
is a straightforward generalization, see details in \cite{FRW, AMSS2}.

A natural method to calculate the motivic Chern class of a locally closed subvariety $X\subset M$ is through the K-theoretic analogue
\begin{equation}\label{tauydef}\mC(X , M)=f_*\left(c^K(T_Y)-\sum_i c^K(T_{D_i})+\sum_{i<j} c^K(T_{D_i\cap D_j})-\dots \right)\end{equation}
of the formula \eqref{csmdef}.

\subsection{Axiomatic characterization of MC classes} \label{MCax}
In certain situations the resolution method (\ref{tauydef}) can be replaced with an axiomatic characterization, which we will now explain in several steps.

\begin{theorem}\cite[Cor 4.5, Lemma 5.1]{FRW}  \label{thm41}
Suppose $M$ has finitely many $G$-orbits.
Then the equivariant MC classes of the orbits satisfy the conditions:
\begin{enumerate}[(i'')]
\item (support condition) the class $\mC(\Omega, M)$ is supported on  $\overline{\Omega}$;
\item (normalization condition) $\mC(\Omega, M)_{|\Omega}=e^K(\nu_\Omega)c^K(T_\Omega)$;
\item (smallness condition) $\Ne(\mC(\Omega, M)_{|\Theta})\subseteq\Ne(\mC(\Theta , M)_{|\Theta})$ [see explanation below];
\item (the divisibility condition) for any $\Theta$ the restriction $\mC(\Omega, M)_{|\Theta}$ is divisible by $c^K(T_\Theta)$ in $K_G(\Theta)[y]$.
\end{enumerate}
\qed
\end{theorem}

In fact, the local condition {\em (iv'')} implies the
condition {\em (i'')}, and it also implies the obvious condition
{\em
\begin{enumerate}[(i'')]
\setcounter{enumi}{4}
\item $\mC(\Omega, M)_{|\Theta}=0$ if $\Theta \not\subset \overline{\Omega}$.
\end{enumerate}
}
\noindent {Note that the K-theoretic version of the condition \eqref{Atiyahcondition} follows from (iii').}

Let us give a precise formulation of the smallness condition {\em (iii'')}. The classes $\mC(\Omega, M)_{|\Theta}$ and $\mC(\Theta , M)_{|\Theta}$ belong to $K_G(\Theta)[y]=R(G_\Theta)[y]$.
Restrict these classes to the representation ring $R(\T_\Theta)[y]$, where $\T_\Theta$ is a maximal torus of $G_\Theta$. For a chosen isomorphism $\T_\Theta=(\C^*)^r$ we have $R(\T_\Theta)[y]=\Z[\alpha_1^{\pm 1},\ldots,\alpha_r^{\pm 1},y]$ where $\alpha_i$ are the K-theoretic Chern roots corresponding to the factors of $(\C^*)^r$. The Newton polygon of $\beta=\sum_I a_I(y)z^I\in R(\T_\Theta)[y]$ (multiindex notation) is 
\[
\Ne(\beta)=\text{convex hull }\{a_I(y)\in \R^r : a_I(y)\not= 0\},
\]
in particular, for this notion the parameter $y$ is considered a constant, not a variable. 
The definition of $\Ne$ depends on the chosen isomorphism between $R(\T_\Theta)[y]$ and $\Z[z_1^{\pm 1},\ldots,z_r^{\pm 1},y]$. Different such choices result in linearly equivalent convex polygons. The smallness condition compares two such Newton polygons, of course the same isomorphism need to be chosen for the two sides.

\begin{example}\label{1trivEX} \rm
If $M=\C$, $G=\C^*$ with the natural action on $\C$ we have a short exact sequence
$$\begin{matrix}0&\longrightarrow &K_{\C^*}(\{0\})&\longrightarrow&K_{\C^*}(\C)&\longrightarrow &K_{\C^*}(\C^*)&\longrightarrow &0\\
&&\parallel&^{(1-\xi^{-1})\cdot}&\parallel&^{ev_{\xi=1}}&\parallel\\
&&\Z[\xi,\xi^{-1}]&&\Z[\xi,\xi^{-1}]& &\Z\end{matrix}$$
In this example the restriction to $\{0\}$ is an isomorphism.
We have
$$\mC(\C^* , \C)=(1+y)\xi^{-1}\,,\qquad \mC(\{0\} , \C)=1-\xi^{-1}$$
and
$$\Ne(\mC(\C^* , \C))_{|0}=\{-1\}\,,\qquad \Ne(\mC(\{0\} , \C))_{|0}=[-1,0].$$
\end{example}

\begin{example}\label{2trivEX}\rm
Consider the standard action of $G={\rm SL}_2(\C)$ on $\C^2$. Then we have
$$\Ne(\mC(\C^2\setminus\{0\} ,\C^2))_{|0}=\Ne(\mC(\{0\} , \C^2))_{|0}=[-1,1]\,.$$
\end{example}

\smallskip

Examples \ref{1trivEX}, \ref{2trivEX} show that the inclusion of Newton polygons in {\em (iii'')} may or may not be strict. However for an interesting class of actions the inclusion is necessarily strict.

\begin{property} (cf. \cite[Definition 4.4]{FRW}) \label{positivity}We say that the action is positive, if for each orbit $\Omega$, $x\in\Omega$ there exists a one dimensional torus $\C^*\hookrightarrow G_x$ which acts on the normal space $(\nu_\Omega)_x$ with positive weights.
\end{property}

Note that Property \ref{positivity} implies that 0 is a vertex of the Newton polygon $\Ne(e^K(\nu_\Omega))$ because
$e^K(\nu_\Omega)_x=\prod_{i=1}^{\codim(\Omega)}(1-\chi_i^{-1})$, where $\chi_i$ are the weights of $\T_\Omega$ acting on $(\nu_\Omega)_x$. In particular $e^K(\nu_\Omega)\neq0$. Let ``+'' denote the Minkowski sum of polygons.

\begin{theorem}\cite[Theorem 5.3]{FRW} \label{tauyaxioms}
Suppose that the action of $G$ on $M$ has Property \ref{positivity}. Then the inclusion in (iii'') is strict. Moreover, for any pair of different orbits $\Theta, \Omega$ we have
\begin{enumerate}[(vi'')]
\item $\Ne(\mC(\Omega, M)_{|\Theta})\;\subseteq\; \Ne(e^K(\nu_\Theta)-1)\;+\;\Ne(c^K(T_\Theta))\;\subsetneq\; \Ne(\mC(\Theta , M)_{|\Theta}).$
\end{enumerate}
\end{theorem}
 A reformulation of condition {\em (vi'')} is that $\Ne(\mC(\Omega, M)_{|\Theta})$ is contained in $\Ne(\mC(\Theta , M)_{|\Theta})$ in such  a way that the origin is a vertex of $\Ne(\mC(\Theta , M)_{|\Theta})$, but it is not contained in $\Ne(\mC(\Omega, M)_{|\Theta})$.

Additionally assuming that the stabilizers $G_\Omega$ are connected guarantees that an element $\beta\in K_\T(M)$ is determined by the restrictions to orbits, and we obtain

\begin{theorem}\cite[Theorem 5.5]{FRW} \label{mainth}
If the action of $G$ on $M$ has finitely many orbits, it is positive, and the stabilizers are connected, then the conditions 
 (ii'), (iv''), (vi'') determine $\mC(\Omega, M)$ uniquely.
\end{theorem}

\begin{example} \rm
The Borel subgroup $B_n$ of $\GL_n(\C)$ acts on the full flag variety $\Fl(4)$ with finitely many orbits. The Newton polygon containment (ie. smallness condition {\em (iii'')}) for two of the orbits, $\Omega=\Omega_{(\{3\},\{4\},\{1\},\{2\})}$ and $\Theta=\Omega_{(\{3\},\{4\},\{2\},\{1\})}$, is illustrated below. For more details about the orbits of this action, their combinatorial codes, and their MC classes see  Section~\ref{sec-local}. The description of the Borel-equivariant K-theory of flag varieties can be found in \cite[\S6]{ChGi} or  in \cite{Uma}. The figure below illustrates (Newton polygons of) the two MC classes, restricted to $K_{B_n}(\Theta)[y]=K_{\T_\Theta}(\pt)[y]$. It turns out that $\T_\Theta=(\C^*)^4$ and hence the two Newton polygons live in $\R^4$. Moreover, both Newton polygons turn out to be contained in a $3$-dimensional subspace (sum of coordinates$=0$). Hence the picture shows 3-dimensional convex polytopes. 

\hfil\obrazek{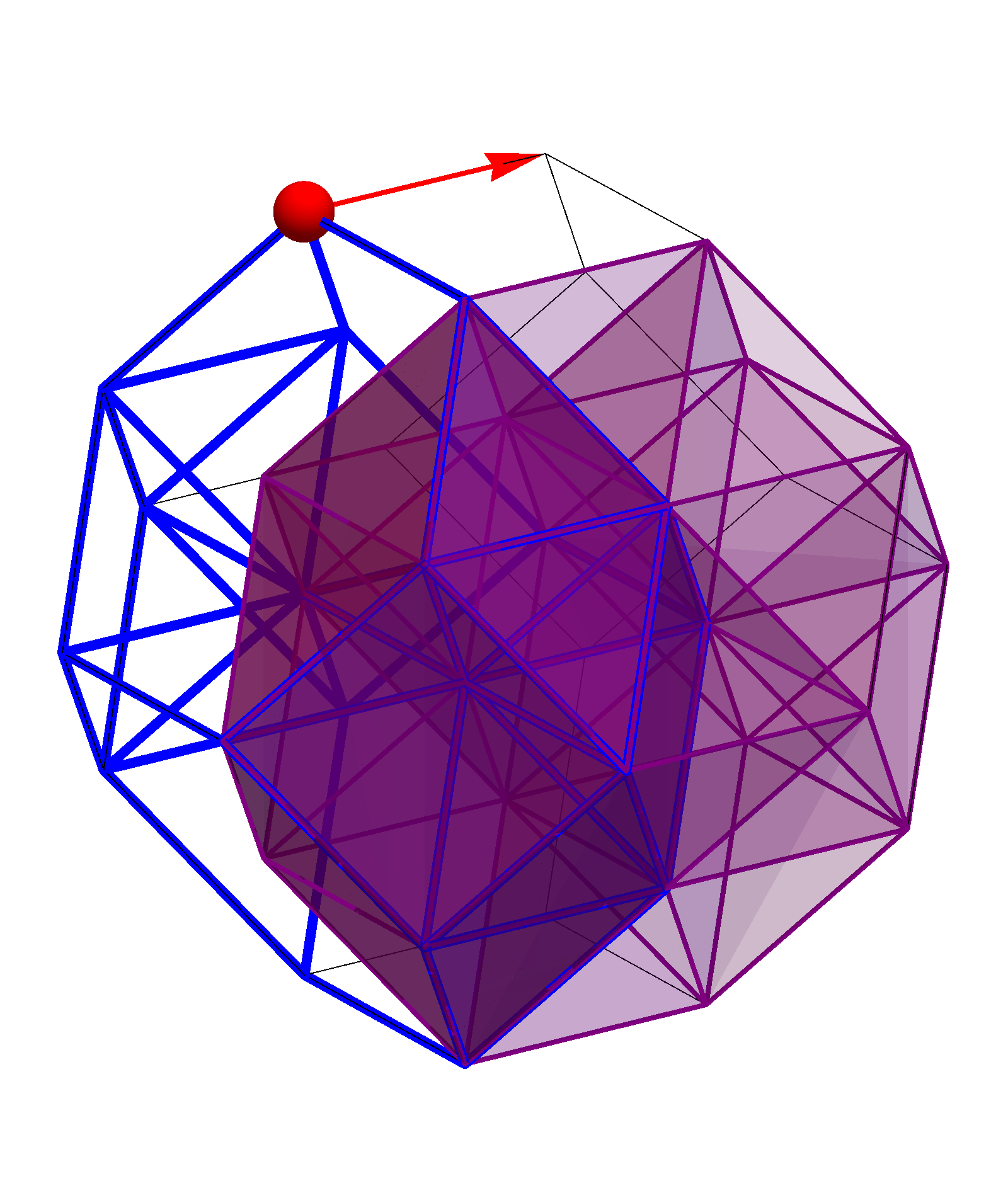}{10cm}

\noindent In blue --- Newton polygon of $e^K(\nu_\Theta)=\lambda_{-1}(\nu^*_\Theta)$.

\noindent Arrow in red --- cotangent weight $T^*_\Theta$.

\noindent In solid violet --- Newton polygon of $\mC(\Omega\cap S_\Theta)$, where $S_\Theta$ is a slice to $\Theta$.

\noindent Edges in violet --- Newton polygon of $\mC(\Omega)$.

\end{example}

\begin{remark} \rm Our smallness conditions {\em (iii'')} and {\em (vi'')} are motivated by the analogous smallness condition for {\em K-theoretic stable envelopes} invented by Okounkov, see \cite[(9.1.10)]{OkounkovK}. There is, however, a difference. In Okounkov's smallness condition the strictness of the inclusion of Newton polygons is guaranteed by the fact that the small Newton polygon can be shifted slightly within the large Newton polygon, see the first picture below. In fact, in the terminology of \cite{OkounkovK}, the inclusion holds for an open set (interior of an alcove) of ``fractional shifts''. However, such a ``wiggle room'' for the inclusions does not necessarily exist for MC classes, see the second and third pictures below. The second picture illustrates the smallness condition {\em (iii'')} (or {\em (vi'')}) for the standard representation $\GL_2(\C)$ on $\C^2$ with $\Omega=\C^2-\{0\}$, $\Theta=\{0\}$. The third picture illustrates smallness condition {\em (iii'')} (or {\em (vi'')}) for the natural $\GL_2(\C)\times \GL_2(\C)$-representation on $\Hom(\C^2,\C^2)$ (a.k.a. $A_2$ quiver representation with dimension vector $(2,2)$) with $\Omega=\{\text{rank 1 maps}\}$ and, $\Theta=\{0\}$.
\begin{center}\obrazek{mobingnewton}{4cm}\hskip45pt\obrazek{gl2newton}{3.5cm}\hskip40pt\obrazek{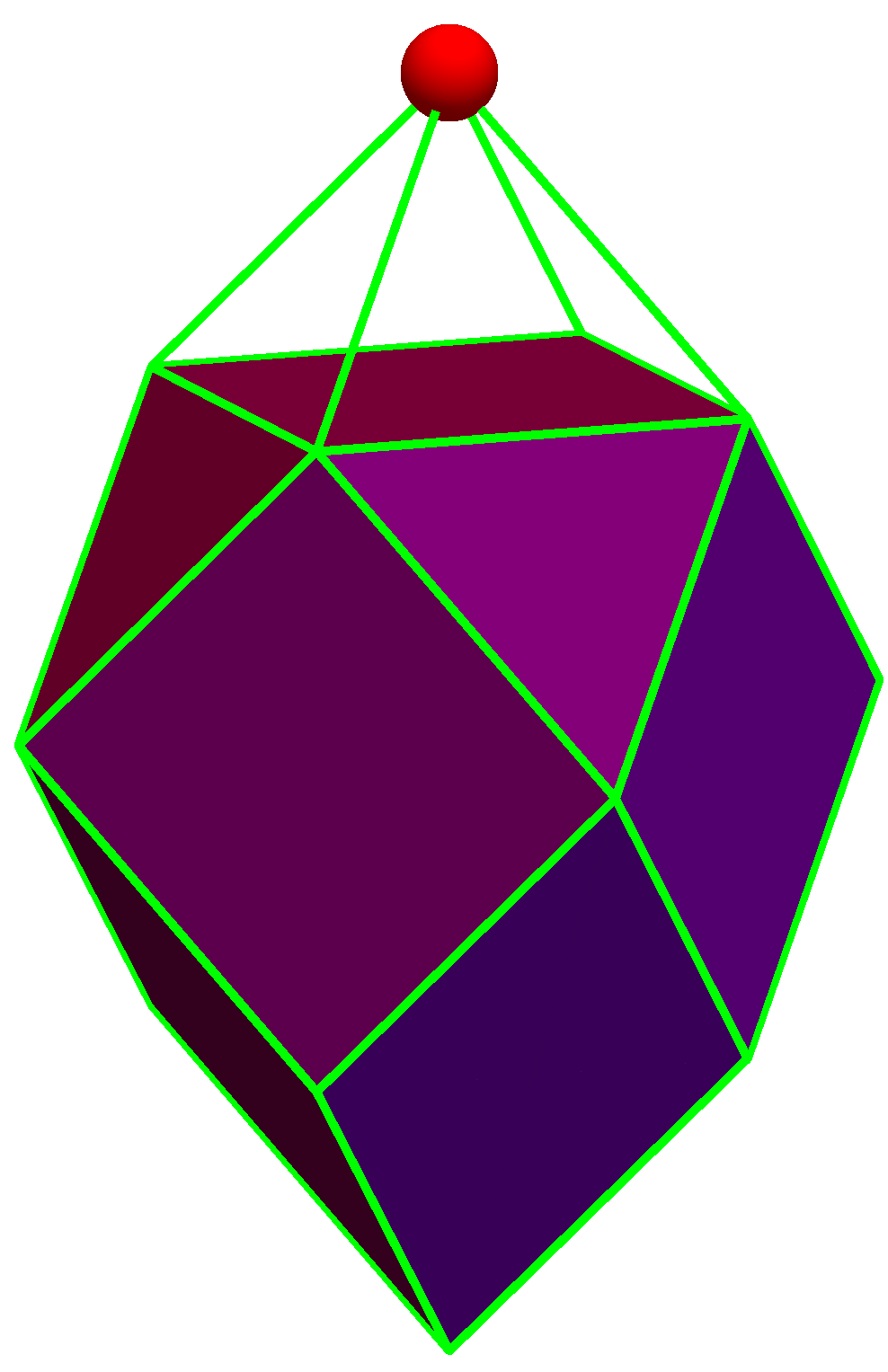}{6cm}\end{center}
\end{remark}

\subsection{The key idea of the proof of Theorem \ref{thm41}}

At the heart of the arguments proving the statements of Section~\ref{MCax} is showing that conditions {\em (iii'')} and {\em (vi'')} hold.
The crucial step in the proof is the special case when $\Theta$ is a point. To test an inclusion of Newton polygons it is enough to restrict an action to each one dimensional torus $\T_0=\C^*\hookrightarrow \T$, with $K_{\T_0}(pt)=\Z[\xi,\xi^{-1}]$. The argument reduces to examining the limit
\begin{equation}\label{limit}\lim_{\xi\to \infty}\frac{\mC(\Omega, M)_\Theta}{e^K(\nu_\Theta)}\,.\end{equation}
We prove that the limit is equal to
$$\chi_y(\Omega\cap M^-_\Theta)\,,$$
where $M^-_\Theta$ is the Bia{\l}ynicki-Birula minus-cell associated to $\Theta$, (see \cite{BB} or \cite[\S4.1]{Ca})
\begin{equation}\label{BB-cell}M^-_\Theta=\{x\in M\;|\;\lim_{t\to\infty} tx\in \Theta\}\,,\end{equation}
and $\chi_y$ is the Hirzebruch $\chi_y$-genus. The proof of this statement in full generality is given in \cite{WeBB} and in the particular cases needed in the proof of
Theorem \ref{mainth} in \cite[Th. 5.3]{FRW}.

Since the limit \eqref{limit} is finite the degree of the denominator is at least the degree of the numerator. This observation leads to the proof of {\em (iii'')}. Moreover, if the action is positive then the limit is equal to $\chi_y(\emptyset)=0$, thus the degree of the denominator must be strictly larger than the degree of the numerator. This observation leads to the proof of {\em (vi'')}.

\subsection{Strict inclusion for homogeneous singularities}\label{sec-alt}

In this section we give a rigorous alternative argument proving the containment property {\em (iii'')} of Newton polygons of $\mC$ classes. We believe that it sheds more light on this intriguing connection between algebraic and convex geometry.

\smallskip

Suppose $\T=\C^*$ acts on a vector space $M=\C^n$ via scalar multiplication. Let $\Omega\subset \C^n$ be an invariant subvariety, where $0\not\in\Omega$. Denote by $Z\subset \P^n$ the projectivization of $\Omega$. We have a diagram
\[
\xymatrix{{\Omega'}\ar[d]_{\simeq}\ar@{^{(}->}[r]& \tilde \C^n \ar[d]^{p} \ar[r]^-{q}
&\P^{n-1}\ar@/^2pc/[l]_{\iota}& Z \ar@{_{(}->}[l]\\
{\Omega}\ar@{^{(}->}[r]& \C^n
}
\]
where $\tilde {\C^n}$ is the blow-up of $\C^n$ at 0 and $\iota:\P^{n-1}\to \tilde{\C^n}$ is the inclusion of the special fiber, ${\Omega'}=q^{-1}(Z)\setminus \iota(Z)$. Let $h=[{\mathcal O}(1)]\in K(\P^{n-1})$. Then
$$\frac{\mC( \Omega', \tilde{\C^n})_{|\P^{n-1}}}{\lambda_{-1}(\nu^*_{\P^{n-1}/\tilde{\C^n}})}=\left( \frac{1+y \xi^{-1} h}{1-\xi^{-1} h}-1\right)\,\mC(Z, \P^{n-1})=\frac{(1+y) \xi^{-1} h}{1-\xi^{-1} h}\;\mC(Z, \P^{n-1})\,.$$
Applying the localization formula for the map $p$ we obtain \begin{equation}\label{eq:blowdown}\frac{\mC(\Omega, \C^n)}{e^K(\C^n)}= p_*\left(\frac{(1+y) \xi^{-1} h}{1-\xi^{-1} h}\;\mC(Z, \P^{n-1})\right)\,.\end{equation}
Let $u=h^{-1}-1\in K(\P^{n-1})$. Note that $u^n=0$.
We have the expansion
\begin{multline*}
\frac{(1+y)\xi^{-1} h}{1-\xi^{-1} h}=\frac{(1+y)\xi^{-1} }{1+u - \xi^{-1} }
=\frac{(1+y)\xi^{-1} }{(1 - \xi^{-1})(1+\tfrac u{1 - \xi^{-1}}) }\\
=\frac{(1+y)\xi^{-1} }{1 - \xi^{-1} }\left( 1-\frac u{1- \xi^{-1} }+\dots+(-1)^{n-1}\frac {u^{n-1}}{(1-\xi^{-1})^{n-1} }\right)\,.\end{multline*}
The expression under the push-forward $p_*$ in (\ref{eq:blowdown}) is of the form
$$\frac1{(1- \xi^{-1})^n}P(u,\xi^{-1},y)\,.$$
The polynomial $P$ is of degree at most $n$ in $\xi^{-1}$, at most $n-1$ in $u$ and it is divisible by $(y+1)\xi^{-1}$. We have\footnote{Since $u^k=[\mathcal O_{\P^{n-1-k}}]\in K(\P^{n-1})$  thus $p_*(u^k)=\chi(\P^{n-1-k};\mathcal O)=1$.}  $p_*(u^k)=1$ for $k<n$.
After the push-forward we obtain
$$\frac{\mC(\Omega, \C^n)}{e^K(\C^n)}=\frac{Q(\xi^{-1},y)}{(1-\xi^{-1})^n}\,,$$
with the polynomial $Q(\xi^{-1},y)=P(1,\xi^{-1},y)$ again divisible by $(y+1)\xi^{-1}$ and of degree at most $n$ in $\xi^{-1}$.

\begin{remark}\rm The argument above shows that $\mC(\Omega, \C^n)$ is divisible by
$(y+1)$. In a similar way this divisibility  can be proven for any quasihomogenous subvariety. Such variety can be presented as a quotient (by a finite group) of homogeneous one.
Application of the Lefschetz-Riemann-Roch formula \cite[Th. 5.1]{CMSS} gives an explicit formula for $\mC(\Omega,\C^n)$.\end{remark}

\subsection{An example: quadratic cone---limits with different choices of 1-parameter subgroup}

This example is based on the computations made in \cite{MiWe}.
Let $\Omega\subset \C^4$ be the open set given by the inequality $z_1z_2-z_3z_4\neq0$.
This is an open orbit of the natural action of $\C^*\times O(4,\C)$ on $\C^4$.

On the other hand, identifying $\C^4$ with $2\times 2$ matrices $\begin{pmatrix}z_1&z_3\\z_4&z_2\end{pmatrix}$ the set $\Omega$ is equal to the set of nondegenerate matrices. Its motivic Chern class was computed in \cite[\S8]{FRW}.

Denote the basis characters of the torus $\T=(\C^*)^3$ by $\alpha,\; \beta,\;\gamma$. Suppose $\T$
acts on $\C^4$  with characters
$$\alpha \beta\; ,\;\alpha /\beta\; ,\;\alpha \gamma\; ,\;\alpha/\gamma\,. $$
 The action preserves the variety $\Omega$. The action has a unique fixed point at  $0$. We study $\Omega\cap(\C^4)_{\{0\}}^-$, the Bia{\l}ynicki-Birula minus-cell (defined by \eqref{BB-cell}) intersected with the orbit,  depending on the choice of the one parameter subgroup.
We apply the formula \cite[Formula 3]{MiWe} for $n=4$, which allows us to compute the motivic Chern class of $\Omega$:
$$\mC(\Omega, \C^4)=(1+y)^2 \left(
\frac{1}{\alpha ^4}y^2+
\left(
 \frac{\beta }{\alpha^3}
+\frac{\gamma }{\alpha^3}
+\frac{1}{\alpha ^3 \beta }
+\frac{1}{\alpha ^3 \gamma }
-\frac{1}{\alpha^2}
-\frac{1}{\alpha ^4}
\right)y
+\frac{1}{\alpha ^2}\right).$$
Below we present a sample of choices of one parameter subgroups and computed limits of motivic Chern classes:
{\renewcommand{\arraystretch}{2}
$$
\begin{array}{|c||c|c|c|c|}\hline
 (\alpha,\beta,\gamma)& \text{minus--cell }(\C^4)_{\{0\}}^-& \text{inequality}&\Omega\cap (\C^4)_{\{0\}}^-&\text{limit of }\frac{\mC(\Omega\cap (\C^4)_{\{0\}}^-)_0}{e^K(\nu_0)}\\
\hline \hline
(\xi, 1,1) & \{0\}
& 0\neq 0 &\emptyset& 0 \\
\hline
(\xi^{-1},1,1) & \C^4 & z_1z_2-z_3 z_4\neq 0 & \begin{matrix}\text{cone complement}\\
[\C^4]-[\C^*][\P^1]^2-[0]\end{matrix}&y^4+(y+1)(1-y)^2-1\\

\hline
 (\xi^{-1},\xi ^2,1) &z_1=0 & z_3z_4\neq 0 & \C\times(\C^*)^2&-y (y+1)^2 \\
\hline
\end{array}
$$
}

\section{Flag manifolds and quiver representation spaces}\label{sec:Fl}

After fixing some combinatorial codes we will define two related geometric objects: a flag variety and a quiver representation space, together with the description of the Borel orbits in these spaces. Then we will present formulas for the motivic Chern classes of the orbits.

\subsection{Combinatorial codes}
Let $n,N$ be non-negative integers, and let $\mu=(\mu_1,\ldots,\mu_N)\in \N^N$ with $\sum_{i=1}^N \mu_i=n$. Denote $\mu^{(i)}=\sum_{j=1}^i \mu_j$.

Let $\I_\mu$ denote the collection of $N$-tuples $(I_1,\ldots,I_N)$ with $I_j\subset \{1,\ldots,n\}$, $I_i\cap I_j=\emptyset$ unless $i=j$, $|I_i|=\mu_i$. For $I\in \I_\mu$ we will use the following notation: $I^{(j)}=\cup_{i=1}^j I_i=\{i^{(j)}_1<\ldots < i^{(j)}_{\mu^{(j)}}\}$.

For $I\in \I_\mu$ define $\ell(I)=|\{(a,b)\in \{1,\ldots,n\}^2:a>b, a\in I_j, b\in I_k, j<k\}|$.

\subsection{Flag variety} \label{sec:flag}
Let $\Fl_\mu$ denote the partial flag variety parameterizing chains of subspaces
$
V_{\bullet}=(0\subset V_1 \subset V_2 \subset \ldots \subset V_{N-1}\subset V_N=\C^n)
$
with $\dim(V_i)=\mu^{(i)}$. Consider the natural action of the Borel subgroup $B_n\subset \GL_n(\C)$ on $\Fl_\mu$. Let $\F_i$ be the tautological bundle over $\Fl_\mu$ whose fiber over the point $(V_\bullet)$ is $V_i$. Let the $B_n$-equivariant K-theoretic Chern roots of $\F_i$ be $\alpha^{(i)}_j$, for $i=1,\ldots,N$, $j=1,\ldots,\mu^{(i)}$, that is, in $K_{B_n}(\Fl_\mu)$ we have $\sum_{j=1}^{\mu^{(i)}} \alpha^{(i)}_j=\F_i$. Then we have
\begin{equation} \label{KFl}
K_{B_n}(\Fl_\mu)[y]=\Z[(\alpha^{(i)}_j)^{\pm 1},y]^{S_{\mu^{(1)}}\times \ldots \times S_{\mu^{(N-1)}}}/(\text{certain ideal}).
\end{equation}
Here the symmetric group $S_{\mu^{(k)}}$ permutes the variables $\alpha^{(k)}_1,\ldots,\alpha^{(k)}_{\mu^{(k)}}$.
\begin{definition}
For $I\in \I_\mu$ define the Schubert cell
\[
\Omega_I=\big\{(V_i)\in \Fl_\mu : \dim(V_p \cap \C^q_{last})=\#\{i\in I^{(p)} : i>n-q\}, \forall p,q\big\} \subset \Fl_\mu,
\]
where $\C^q_{last}$ is the span of the last $q$ standard basis vectors in $\C^n$. Its codimension in $\Fl_\mu$ is $\ell(I)$.
\end{definition}

\subsection{Quiver representation space.}
Consider $\RRep_\mu=\oplus_{j=1}^{N-1} \Hom(\C^{\mu^{(j)}}, \C^{\mu^{(j+1)}})$ with the action of
\[
\GL\times B_n=\prod_{i=1}^{N-1} \GL_{\mu^{(i)}}(\C) \times B_n
\]
given by
\[
(g_i,g_N)_{i=1,\ldots,N-1} \cdot (a_j)_{j=1,\ldots,N-1} = ( g_{j+1} \circ a_j \circ g_j^{-1} )_{j=1,\ldots,N-1}.
\]
For $i=1,\ldots,N$ let $\F_i$ be the tautological rank $\mu^{(i)}$ bundle over the classifying space of the $i$'th component of $\GL\times B_n$. Let $\alpha^{(i)}_j$ be the K-theoretic Chern roots of $\F_i$, for $i=1,\ldots,N$, $j=1,\ldots, \mu^{(i)}$. Then we have
\begin{equation}\label{KRep}
K_{\GL\times B_n}(\RRep_\mu)[y]=\Z[(\alpha^{(i)}_j)^{\pm 1},y]^{S_{\mu^{(1)}}\times \ldots \times S_{\mu^{(N-1)}}}.
\end{equation}

\medskip

Notice that the K-theory algebra in (\ref{KFl}) is a quotient of the K-theory algebra in (\ref{KRep}).

\begin{definition}
For $I\in \I_\mu$ define the matrix Schubert cell $M\Omega_I \subset \RRep_\mu$ by
\begin{multline*}
M\Omega_I=\{(a_i)\in \RRep_\mu : a_i \text{ is injective }\forall i,
\\
\dim( (a_{N-1}\circ \ldots \circ a_p)(\C^{\mu^{(p)}})  \cap \C^q_{last})=\#\{i\in I^{(p)} : i>n-q\}, \forall p,q\},
\end{multline*}
where $\C^q_{last}$ is the span of the last $q$ standard basis vectors in $\C^n$. Its codimension in $\RRep_\mu$ is  $\ell(I)$.
\end{definition}

\begin{remark} \label{rem:GIT}
Of course we have that $\Fl_\mu=\RRep_\mu \sslash\GL$ as a GIT quotient space. Namely, let $\RRep_\mu^{ss}$ be the subset consisting only injective maps. Then the natural map
\[
\RRep_\mu^{ss} \to \Fl_\mu, \qquad (a_i) \mapsto ( (a_{N-1}\circ \ldots \circ a_p)(\C^{\mu^{(p)}}) )_p
\]
is a topological quotient by $\GL$, in fact $\RRep_\mu^{ss}\to \Fl_\mu$ is a principal bundle, cf. \cite{Zie}. Under this map, we have $M\Omega_I$ maps to $\Omega_I$.
This correspondence between $\Fl_\mu$ and $\RRep_\mu$ will be used in Section \ref{sec-laci}.
\end{remark}

\subsection{Weight functions} \label{sec:weight}
In this section we define some explicit functions that will be used in naming the equivariant motivic Chern classes of Schubert and matrix Schubert cells. For more details on these functions see e.g. \cite{RTV, RTVselecta} or references therein.

 For $I\in \I_\mu$, $j=1,\ldots,N-1$, $a=1,\ldots,\mu^{(j)}$, $b=1,\ldots,\mu^{(j+1)}$ define
\[
\psi_{I,j,a,b}(\xi)=\begin{cases}
1-\xi & \text{if } i^{(j+1)}_b<i^{(j)}_a \\
 (1+y)\xi & \text{if }  i^{(j+1)}_b=i^{(j)}_a \\
1+y\xi & \text{if }  i^{(j+1)}_b>i^{(j)}_a.
\end{cases}
\]
Define the {\em weight function}
\[
W_I=
\Sym_{S_{\mu^{(1)}} \times \ldots \times S_{\mu^{(N-1)}}} U_I
\]
where
\[
U_I=
\prod_{j=1}^{N-1} \prod_{a=1}^{\mu^{(j)}} \prod_{b=1}^{\mu^{(j+1)}}
\psi_{I,j,a,b}(\alpha^{(j)}_a/\alpha^{(j+1)}_b )
\cdot
\prod_{j=1}^{N-1} \prod_{1\leq a <  b\leq \mu^{(j)}} \frac{ 1+y\alpha^{(j)}_b/\alpha^{(j)}_a}{1-\alpha^{(j)}_b/\alpha^{(j)}_a}.
\]
Here the symmetrizing operator is defined by
\[
\Sym_{S_{\mu^{(1)}}  \times \ldots \times S_{\mu^{(N-1)}}} U_I=
\sum_{\sigma\in S_{\mu^{(1)}} \times \ldots \times S_{\mu^{(N-1)}}} U_I(\sigma(\alpha^{(j)}_a))
\]
where the $j$th component of $\sigma$ (an element of $S_{\mu^{(j)}}$) permutes the $\alpha^{(j)}$ variables.
For
\[
c_\mu=
\prod_{j=1}^{N-1}
\prod_{a=1}^{\mu^{(j)}}
\prod_{b=1}^{\mu^{(j)}}
(1+y\alpha^{(j)}_b/\alpha^{(j)}_a),
\qquad
c'_\mu=
\prod_{j=1}^{N-1}
\prod_{a=1}^{\mu^{(j+1)}}
\prod_{b=1}^{\mu^{(j)}}
(1+y\alpha^{(j)}_b/\alpha^{(j+1)}_a)
\]
define the {\em modified weight functions}
\[
\tilde{W}_I=W_I/c_\mu, \qquad \hat{W}_I=W_I/c'_\mu.
\]

Observe that $\tilde{W}_I, \hat{W}_I$ are not Laurent polynomials, but rather ratios of two such.

Weight functions were defined by Tarasov and Varchenko in relation with hypergeometric solutions to qKZ differential equations, see e.g. \cite{TV}.

\subsection{Motivic Chern classes of Schubert cells given by weight functions}

\begin{theorem}\label{MCmain}
For the motivic Chern classes of matrix Schubert and ordinary Schubert cells we have
\begin{align}
\label{egy} \mC(M\Omega_I, \RRep_\mu)=W_I \qquad & \in K_{\GL\times B_n}(\RRep_\mu)[y], \\
\label{ketto} \mC(\Omega_I, \Fl_\mu)=[\tilde{W}_I] \qquad & \in K_{ B_n}(\Fl_\mu)[y].
\end{align}
\end{theorem}

First let us comment on the two statements of the theorem. As we saw in (\ref{KRep}) the ring $K_{\GL\times B_n}(\RRep_\mu)[y]$ is a Laurent polynomial ring, and claim (\ref{egy}) states which Laurent polynomial is the sought MC class. However, the ring $K_{\GL\times B_n}(\Fl_\mu)[y]$ is a quotient ring of that Laurent polynomial ring, cf. (\ref{KFl}), hence (\ref{ketto}) only names {\em a representative} of the sought MC class. Moreover, the function $\tilde{W}_I$ is a rational function, so the equality (\ref{ketto}) is meant in the following sense: the restriction to  each torus fixed of the two sides of (\ref{ketto}) are the same---in particular the fix point restrictions of $\tilde{W}_I$ are Laurent polynomials.

\begin{remark}\label{segre}
If we define the {\em motivic Segre class} of $X \subset M$ by
\[
\mS(X , M)=\mC(X , M)/c^K(TM),
\]
then we can rephrase Theorem \ref{MCmain} in the more symmetric form
\[
\mS(M\Omega_I, \RRep_\mu)=\hat{W}_I, \qquad
\mS(\Omega_I, \Fl_\mu)=[\hat{W}_I].
\]
This is due to the fact that $c^K_{\GL\times B_n}(T\RRep_\mu)=c'_\mu$, and $[c_{B_
n}^K(T\Fl_\mu)]=c'_\mu/c_\mu$. That is, $c_\mu$ is the K-theoretic total Chern class of the fibers of the GIT quotient map of Remark \ref{rem:GIT}.
\end{remark}

\bigskip

Now let us review different strategies that can be used to prove Theorem \ref{MCmain}---some of them already present in the literature.
One can
\begin{enumerate}[(a)]
\item prove (\ref{egy}) by resolution of singularities;
\item prove (\ref{egy}) by the interpolation Theorem \ref{mainth};
\item prove (\ref{ketto}) by resolution of singularities;
\item prove (\ref{ketto}) by the interpolation Theorem \ref{mainth};
\item prove that (\ref{egy}) implies (\ref{ketto}) via Remark \ref{segre}.
\end{enumerate}

Carrying out (b) and/or (d) has the advantage of not having to construct a resolution with normal crossing divisors. Carrying out (a) and/or (c) has the advantage of not having to prove the sophisticated Newton polygon properties of the weight functions.

Historically, first the Newton polygon properties of weight functions were proved (in the context of K-theoretic stable envelopes), see \cite[Section 3.5]{RTVselecta}. Those proofs are complicated algebraic arguments that were unknown before even by the experts of weight functions. Nevertheless, those arguments provide a complete proof of Theorem \ref{MCmain}, see details in \cite{FRW}.

It is remarkable that resolutions for $\overline{M\Omega}_I$, $\overline{\Omega}_I$ can be constructed from which the motivic Chern classes are calculated easily. More importantly, for well chosen resolutions the formulas obtained for $\mC(M\Omega_I, \RRep_\mu)$,  $\mC(\Omega_I, \Fl_\mu)$ are {\em exactly the defining formulas of weight functions}. As a result, such an argument  {\em reproves} Theorem \ref{MCmain}, and through that, the Newton polygon properties of weight functions without calculation. In hindsight, such an argument would be a more natural proof of Theorem \ref{MCmain}. Since formally proving Theorem \ref{MCmain} is not needed anymore (as we mentioned above, \cite[Section 3.5]{RTVselecta} provides a proof) we will not present a proof based on resolution of singularities in detail. Also, such an argument based on resolution is implicitly present in \cite{csmCOHA} and an analogous argument is explicitly presented in the elliptic setting in \cite{RWell}. As a final remark let us mention our appreciation of the creativity of those who defined weight functions without seeing the resolution calculation for  $\mC(M\Omega_I, \RRep_\mu)$ and $\mC(\Omega_I, \Fl_\mu)$!

\medskip

The implication (e) is not formally needed to complete the proof of Theorem \ref{MCmain}. Yet, we find it important enough, that the we include the proof of the general such statement (namely motivic Chern classes before vs after GIT quotient) in Section \ref{sec-laci}.

\section{MC classes of Schubert cells in full flag manifolds: positivity and log-concavity}\label{sec:pos}

Characteristic classes of singularities often display positivity properties. Motivic Chern classes are not exceptions: in this section we present three conjectures on the signs (and concavity) of coefficients of MC classes in certain expansions.

\subsection{Positivity}\label{sec-p1}

For $\mu=(1,1,\ldots,1)$ ($n$ times) the space $\Fl_\mu$ is the full flag variety, let us rename it to $\Fl(n)$. Recall the presentation of its equivariant K-theory algebra from (\ref{KFl}). For brevity let us rename the ``last'' set of variables, the ``equivariant variables'' to  $\tau_i:=\alpha^{(n)}_i$.

The traditional geometric basis of $K_{B_n}(\Fl(n))[y]$ consists of the classes of the structure sheaves of the Schubert varieties (the closures of the Schubert cells). For $w=(w(1),\ldots,w(n))\in S_n$ let ${\bf [w]}$ denote the class of the structure sheaf of the closure of $\Omega_{(\{w(1)\},\{w(2)\},\ldots,\{w(n)\})}$. Also, set $\mC[w]=\mC(\Omega_{(\{w(1)\},\{w(2)\},\ldots,\{w(n)\})})$.

Our Theorem \ref{MCmain} can be used to expand the MC classes of orbits in the traditional basis of structure sheaves of Schubert varieties. For $n=3$ we obtain the following expansions.

{\def\ffff#1{{\,\bf[#1]}\hfill\break\phantom{\mC XXX}}
\def\fffg#1{{\,\bf[#1]}\break\phantom{\mC XXX}}
\def\fffglast#1{{\,\bf[#1].}}
\def\fffO{\hfill\break\phantom{\mC XXX}\hfill}

$\mC[{1,2,3}] = \left(\frac{\tau _1^2 }{\tau _3^2}y^3+\left(\frac{\tau _1^2}{\tau _2 \tau _3}+\frac{\tau _1}{\tau _3}+\frac{\tau _2 \tau _1}{\tau _3^2}\right) y^2+\left(\frac{\tau _1}{\tau _2}+\frac{\tau _1}{\tau _3}+\frac{\tau _2}{\tau _3}\right) y+1\right) \ffff{1,2,3}-\left(\left(\frac{\tau _1^2}{\tau _2 \tau _3}+\frac{\tau _1^2}{\tau _3^2}\right) y^3+\left(\frac{\tau _1^2}{\tau _2 \tau _3}+\frac{\tau _1}{\tau _2}+\frac{2 \tau _1}{\tau _3}+\frac{\tau _2 \tau _1}{\tau _3^2}\right) y^2+\left(\frac{\tau _1}{\tau _2}+\frac{\tau _1}{\tau _3}+\frac{\tau _2}{\tau _3}+1\right) y+1\right) \ffff{1,3,2}-\left(\left(\frac{\tau _1^2}{\tau _3^2}+\frac{\tau _2 \tau _1}{\tau _3^2}\right) y^3+\left(\frac{\tau _1^2}{\tau _2 \tau _3}+\frac{2 \tau _1}{\tau _3}+\frac{\tau _2 \tau _1}{\tau _3^2}+\frac{\tau _2}{\tau _3}\right) y^2+\left(\frac{\tau _1}{\tau _2}+\frac{\tau _1}{\tau _3}+\frac{\tau _2}{\tau _3}+1\right) y+1\right) \ffff{2,1,3}+\left(\left(\frac{\tau _1^2}{\tau _2 \tau _3}+\frac{\tau _1^2}{\tau _3^2}+\frac{\tau _1}{\tau _3}+\frac{\tau _2 \tau _1}{\tau _3^2}+\frac{\tau _2}{\tau _3}\right) y^3+\left(\frac{\tau _1^2}{\tau _2 \tau _3}+\frac{\tau _1}{\tau _2}+\frac{3 \tau _1}{\tau _3}+\frac{\tau _2 \tau _1}{\tau _3^2}+\frac{2 \tau _2}{\tau _3}+1\right) y^2\right.\fffO{}\left.+\left(\frac{\tau _1}{\tau _2}+\frac{\tau _1}{\tau _3}+\frac{\tau _2}{\tau _3}+2\right) y+1\right) \fffg{2,3,1}+\left(\left(\frac{\tau _1^2}{\tau _2 \tau _3}+\frac{\tau _1^2}{\tau _3^2}+\frac{\tau _1}{\tau _2}+\frac{\tau _1}{\tau _3}+\frac{\tau _2 \tau _1}{\tau _3^2}\right) y^3+\left(\frac{\tau _1^2}{\tau _2 \tau _3}+\frac{2 \tau _1}{\tau _2}+\frac{3 \tau _1}{\tau _3}+\frac{\tau _2 \tau _1}{\tau _3^2}+\frac{\tau _2}{\tau _3}+1\right) y^2\right.\fffO{}\left.+\left(\frac{\tau _1}{\tau _2}+\frac{\tau _1}{\tau _3}+\frac{\tau _2}{\tau _3}+2\right) y+1\right) \fffg{3,1,2}-\left(\left(\frac{\tau _1^2}{\tau _2 \tau _3}+\frac{\tau _1^2}{\tau _3^2}+\frac{\tau _1}{\tau _2}+\frac{2 \tau _1}{\tau _3}+\frac{\tau _2 \tau _1}{\tau _3^2}+\frac{\tau _2}{\tau _3}+1\right) y^3+\left(\frac{\tau _1^2}{\tau _2 \tau _3}+\frac{2 \tau _1}{\tau _2}+\frac{3 \tau _1}{\tau _3}+\frac{\tau _2 \tau _1}{\tau _3^2}+\frac{2 \tau _2}{\tau _3}+2\right) y^2\right.\fffO{}\left.+\left(\frac{\tau _1}{\tau _2}+\frac{\tau _1}{\tau _3}+\frac{\tau _2}{\tau _3}+2\right) y+1\right) \fffglast{3,2,1}$

\def\ffff#1{{\,\bf[#1]}\hfill\break\phantom{\mC({1,2,3})=}}
\def\fffflast#1{{\,\bf[#1].}}

$\mC[{1,3,2}] = \left(\frac{\tau _1^2 }{\tau _2 \tau _3}y^2+\left(\frac{\tau _1}{\tau _2}+\frac{\tau _1}{\tau _3}\right) y+1\right) \ffff{1,3,2}-\left(\left(\frac{\tau _1^2}{\tau _2 \tau _3}+\frac{\tau _1}{\tau _3}+\frac{\tau _2}{\tau _3}\right) y^2+\left(\frac{\tau _1}{\tau _2}+\frac{\tau _1}{\tau _3}+\frac{\tau _2}{\tau _3}+1\right) y+1\right) \ffff{2,3,1}-\left(\left(\frac{\tau _1^2}{\tau _2 \tau _3}+\frac{\tau _1}{\tau _2}\right) y^2+\left(\frac{\tau _1}{\tau _2}+\frac{\tau _1}{\tau _3}+1\right) y+1\right) \ffff{3,1,2}+\left(\left(\frac{\tau _1^2}{\tau _2 \tau _3}+\frac{\tau _1}{\tau _2}+\frac{\tau _1}{\tau _3}+\frac{\tau _2}{\tau _3}+1\right) y^2+\left(\frac{\tau _1}{\tau _2}+\frac{\tau _1}{\tau _3}+\frac{\tau _2}{\tau _3}+2\right) y+1\right) \fffflast{3,2,1}$

$\mC[{2,1,3}] = \left(\frac{\tau _1 \tau _2 }{\tau _3^2}y^2+\left(\frac{\tau _1}{\tau _3}+\frac{\tau _2}{\tau _3}\right) y+1\right) \ffff{2,1,3}-\left(\left(\frac{\tau _2}{\tau _3}+\frac{\tau _1 \tau _2}{\tau _3^2}\right) y^2+\left(\frac{\tau _1}{\tau _3}+\frac{\tau _2}{\tau _3}+1\right) y+1\right) \ffff{2,3,1}-\left(\left(\frac{\tau _1}{\tau _2}+\frac{\tau _1}{\tau _3}+\frac{\tau _2 \tau _1}{\tau _3^2}\right) y^2+\left(\frac{\tau _1}{\tau _2}+\frac{\tau _1}{\tau _3}+\frac{\tau _2}{\tau _3}+1\right) y+1\right) \ffff{3,1,2}+\left(\left(\frac{\tau _1}{\tau _2}+\frac{\tau _1}{\tau _3}+\frac{\tau _2 \tau _1}{\tau _3^2}+\frac{\tau _2}{\tau _3}+1\right) y^2+\left(\frac{\tau _1}{\tau _2}+\frac{\tau _1}{\tau _3}+\frac{\tau _2}{\tau _3}+2\right) y+1\right) \fffflast{3,2,1}$

$\mC[{2,3,1}] = \left(\frac{\tau _2 }{\tau _3}y+1\right) \ffff{2,3,1}-\left(\left(\frac{\tau _2}{\tau _3}+1\right) y+1\right) \fffflast{3,2,1}$

$\mC[{3,1,2}] = \left(\frac{\tau _1 }{\tau _2}y+1\right) \ffff{3,1,2}-\left(\left(\frac{\tau _1}{\tau _2}+1\right) y+1\right) \fffflast{3,2,1}$

$\mC[{3,2,1}] = \fffflast{3,2,1}$}

\smallskip
The following conjecture can be verified in the formulas above, and we also verified it for larger flag varieties ($n\leq 5$).

\begin{conjecture}\label{posconj}
Let $p, w\in S_n$. The coefficient of ${\bf [w]}$ in the expansion of $\mC[p]$ is a Laurent polynomial in $\tau_1,\ldots,\tau_n,y$ whose terms have sign
$(-1)^{\ell(p)-\ell(w)}$.
\end{conjecture}

We have been informed by the authors of \cite{AMSS2} that they also observed the sign behavior described in Conjecture 6.1.

\subsection{Log concavity}
To illustrate a new feature of the coefficients of the ${\bf [w]}$-expansions, let us make the substitution $\tau_i=1$ ($\forall i$), that is, consider the {\em non-equivariant} motivic Chern classes of the Schubert cells. For $n=4$ and for the open cell we obtain
{\def\fff#1{{\,\bf [#1]}\hfill\break\phantom{\mC[{1,2,3,4}]=}}
{\def\ffflast#1{{\,\bf [#1].}}
\vskip6pt
$\mC[{1,2,3,4}]=
\,(y+1)^6	\fff{1,2,3,4}
-(y+1)^5 (2 y+1)	\fff{1,2,4,3}
-(y+1)^5 (2 y+1)	\fff{1,3,2,4}
-(y+1)^5 (2 y+1)	\fff{2,1,3,4}
+(y+1)^4 (5 y^2+4 y+1)	\fff{1,3,4,2}
+(y+1)^4 (5 y^2+4 y+1)	\fff{1,4,2,3}
+(y+1)^4 (2 y+1)^2	\fff{2,1,4,3}
+(y+1)^4 (5 y^2+4 y+1)	\fff{2,3,1,4}
+(y+1)^4 (5 y^2+4 y+1)	\fff{3,1,2,4}
-(y+1)^3 (8 y^3+11 y^2+5 y+1)	\fff{1,4,3,2}
-(y+1)^3 (14 y^3+14 y^2+6 y+1)	\fff{2,3,4,1}
-(y+1)^3 (13 y^3+13 y^2+6 y+1)	\fff{2,4,1,3}
-(y+1)^3 (2 y+1) (6 y^2+4 y+1)	\fff{3,1,4,2}
-(y+1)^3 (8 y^3+11 y^2+5 y+1)	\fff{3,2,1,4}
-(y+1)^3 (14 y^3+14 y^2+6 y+1)	\fff{4,1,2,3}
+(y+1)^2 (24 y^4+36 y^3+21 y^2+7 y+1)	\fff{2,4,3,1}
+(y+1)^2 (23 y^4+35 y^3+22 y^2+7 y+1)	\fff{3,2,4,1}
+(y+1)^2 (3 y+1) (10 y^3+10 y^2+4 y+1)	\fff{3,4,1,2}
+(y+1)^2 (23 y^4+35 y^3+22 y^2+7 y+1)	\fff{4,1,3,2}
+(y+1)^2 (24 y^4+36 y^3+21 y^2+7 y+1)	\fff{4,2,1,3}
-(y+1) (44 y^5+85 y^4+66 y^3+29 y^2+8 y+1)	\fff{3,4,2,1}
-(y+1) (49 y^5+91 y^4+69 y^3+30 y^2+8 y+1)	\fff{4,2,3,1}
-(y+1) (44 y^5+85 y^4+66 y^3+29 y^2+8 y+1)	\fff{4,3,1,2}
+(64 y^6+163 y^5+169 y^4+98 y^3+37 y^2+9 y+1)	\fff{4,3,2,1}
$

\noindent In the expression above the permutations are ordered primary by length, secondary by lexicographical order.
\medskip

\noindent The coefficient of ${\bf[5,4,3,2,1]}$ in $\mC{[1, 2, 3, 4,5]}$ is
\[
1 + 14 y + 92 y^2 + 377 y^3 + 1120 y^4 + 2630 y^5 + 4972 y^6 +
 7148 y^7 + 7024 y^8 + 4063 y^9 + 1024 y^{10}.
\]

\begin{definition} The polynomial $\sum_{k=0}^d a_ky^k$ is said to be log-concave if
$a_k^2\geq a_{k-1}a_{k+1}$ for
$0<k<d$. It is strictly log-concave if the equality is strict.
\end{definition}

See Huh's survey article \cite{Huh} for the role of log-concavity in geometry and  combinatorics.

\begin{conjecture}\label{log_con_conj}
The coefficients in the ${\bf [w]}$-expansion of the non-equivariant $\mC[p]$ classes are strictly log-concave.
\end{conjecture}

We checked  Conjecture \ref{log_con_conj} for $n\leq 6$.
\subsection{Positivity in new variables}\label{sec-p2}

There is another kind of positivity which does not follow from the conjectured properties above. Namely, let us substitute $\frac{\tau_i}{\tau_{i+1}}=s_i+1$ and $y=-1-\delta$.
For example
$$\mC[{4,3,1,2}] = (1+ y \tfrac{\tau_1}{\tau_2}) {\bf[{4,3,1,2}]}-y {\bf[{4,3,2,1}]}
=  -(s_1 +\delta(1+  s_1)) {\bf[{4,3,1,2}]}+(1+\delta) {\bf[4,3,2,1]}\,.$$
The coefficient of $\bf[4,3,2,1]$ in $\mC[1,4,3,2]$ in the new variables equals
\bigskip

$(1 + s_1)^3 (1 + s_2)^2 (1 + s_3)+$

$
 \delta (9 + 17 s_1 + 12 s_1^2 + 3 s_1^3 + 14 s_2 + 28 s_1 s_2 + 22 s_1^2 s_2 +
    6 s_1^3 s_2 + 6 s_2^2 + 12 s_1 s_2^2 + 10 s_1^2 s_2^2 + 3 s_1^3 s_2^2 +
    8 s_3 + 15 s_1 s_3 + 11 s_1^2 s_3 + 3 s_1^3 s_3 + 13 s_2 s_3 +
    26 s_1 s_2 s_3 + 21 s_1^2 s_2 s_3 + 6 s_1^3 s_2 s_3 + 6 s_2^2 s_3 +
    12 s_1 s_2^2 s_3 + 10 s_1^2 s_2^2 s_3 + 3 s_1^3 s_2^2 s_3)+$

$
 \delta^2 (21 + 28 s_1 + 15 s_1^2 + 3 s_1^3 + 26 s_2 + 40 s_1 s_2 + 26 s_1^2 s_2 +
    6 s_1^3 s_2 + 9 s_2^2 + 15 s_1 s_2^2 + 11 s_1^2 s_2^2 + 3 s_1^3 s_2^2 +
    16 s_3 + 22 s_1 s_3 + 13 s_1^2 s_3 + 3 s_1^3 s_3 + 22 s_2 s_3 +
    35 s_1 s_2 s_3 + 24 s_1^2 s_2 s_3 + 6 s_1^3 s_2 s_3 + 9 s_2^2 s_3 +
    15 s_1 s_2^2 s_3 + 11 s_1^2 s_2^2 s_3 + 3 s_1^3 s_2^2 s_3)+$

$\delta^3 (14 + 14 s_1 + 6 s_1^2 + s_1^3 + 14 s_2 + 18 s_1 s_2 + 10 s_1^2 s_2 +
    2 s_1^3 s_2 + 4 s_2^2 + 6 s_1 s_2^2 + 4 s_1^2 s_2^2 + s_1^3 s_2^2 + 9 s_3 +
    10 s_1 s_3 + 5 s_1^2 s_3 + s_1^3 s_3 + 11 s_2 s_3 + 15 s_1 s_2 s_3 +
    9 s_1^2 s_2 s_3 + 2 s_1^3 s_2 s_3 + 4 s_2^2 s_3 + 6 s_1 s_2^2 s_3 +
    4 s_1^2 s_2^2 s_3 + s_1^3 s_2^2 s_3).$

\begin{conjecture}\label{conj:delta}
The coefficients of the $s_i,\delta$-monomials in the ${\bf [w]}$-expansion of the $\mC[p]$ classes have sign $(-1)^{\ell(w)}$. Note that here the sign depends only on the length the permutation $w$, not on the length of $p$, as it was in Conjecture \ref{posconj}.
\end{conjecture}

\begin{remark}
The positivity properties of Sections \ref{sec-p1} and \ref{sec-p2}  imply positivity properties of the restrictions of MC classes to fixed points. These ``local'' positivity properties are another instances of the positivity in the $\delta$-variables discussed in \cite[\S15]{WeSelecta}. It is implied by the Conjecture \ref{conj:delta} and \cite{AGM}.
\end{remark}

\section{Local picture}\label{sec-local}
To illustrate the pretty convex geometric objects encoded by motivic Chern classes in this section we present some pictures. Consider $\Fl(3)$ and its six Schubert cells parameterized by permutations. According to Theorem \ref{MCmain} the six fixed point restrictions of these classes satisfy some strict containment properties. Below we present the Newton polygons (and some related weights) of the six fixed point restrictions (rows) of the six Schubert cells (columns). The Newton polygons live in the $\tau_1+\tau_2+\tau_3=0$ plane of $\R^3$ (with coordinates $\tau_i$), hence instead of 3D pictures we only draw the mentioned plane.

\begin{center}\obrazek{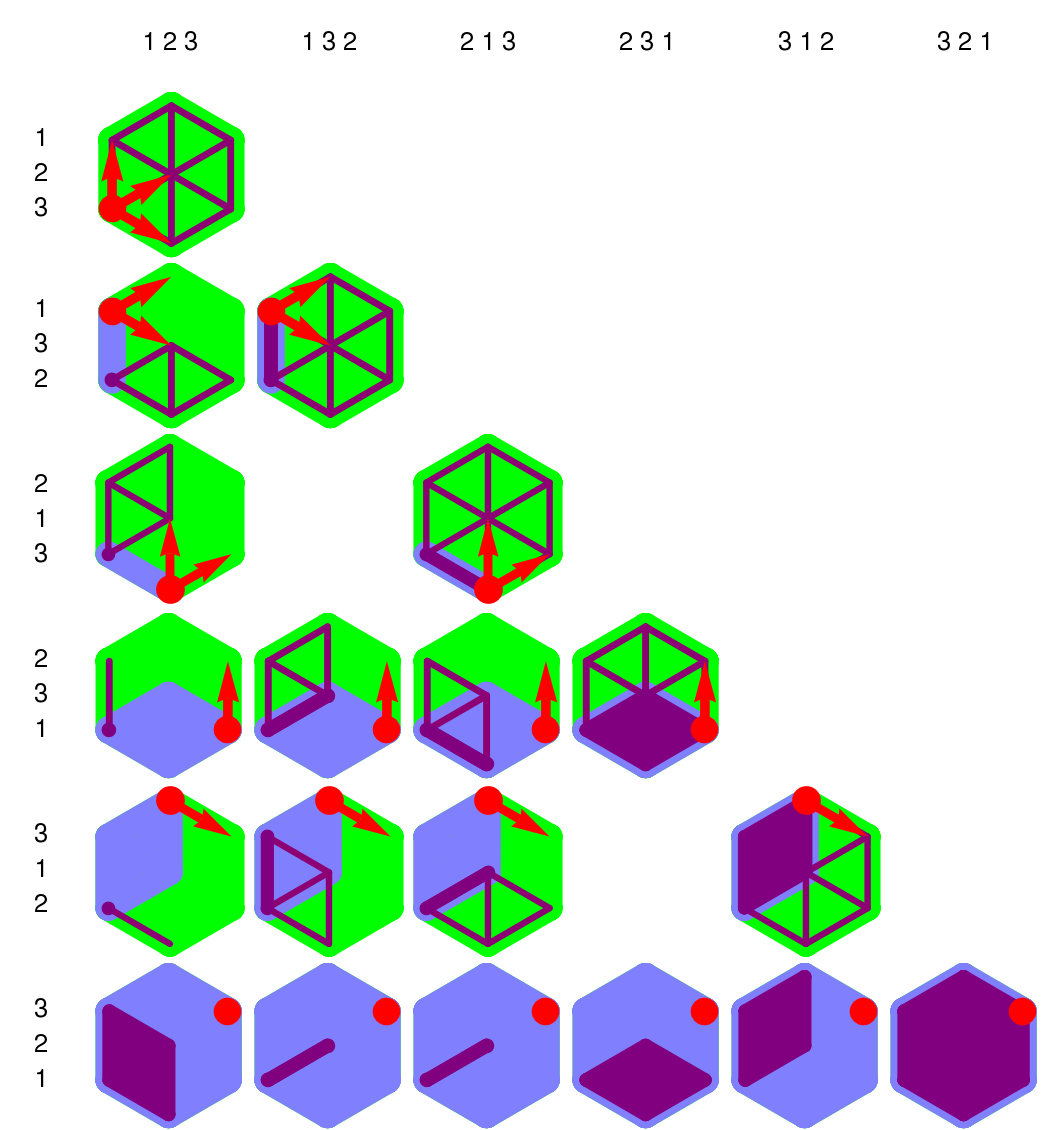}{12cm}\end{center}

In  blue --- Newton polygon of $e^K(\nu_\Theta)=\lambda_{-1}(\nu^*_\Theta)$.

In red --- cotangent weight $T^*_\Theta$.

In violet --- Newton polygon of $\mC(\Omega\cap S_\Theta)$, where $S_\Theta$ is a slice to the orbit.

Edges in violet --- Newton polygon of $\mC(\Omega)$.

\section{Transversality and GIT quotients} \label{sec-laci}

Both CSM and MC classes have their Segre version (cf. Remark \ref{segre}): for example for the K-theory case we have
\[ \mS(f:Z\to M):=\mC(f)/\lambda_y(T^*M).\]
\subsection{Transversality of motivic Segre class}\label{sec:trans}
Motivic Segre classes behave well for transversal pull-back, for a fine notion of transversality.
To formulate this notion let us recall the definition of the motivic Chern class given in \cite{FRW}. 
Consider the case of $\mC(f)$ where $f:U\to M$ is an inclusion of smooth varieties, but the inclusion is not necessarily closed (or equivalently, proper).
\begin{definition} A proper normal crossing extension (PNC) of $f:U \to M$ is a morphism $\bar f:Y\to M$ and an inclusion $j:U\hookrightarrow Y$ such that
\begin{enumerate}
  \item $f=\bar f\circ j$,
  \item $Y$ is smooth
  \item $\bar f$ is proper
  \item The exceptional divisor $D:=Y\setminus j(U)=\bigcup_{i=1}^sD_i$  is a simple normal crossing divisor (i.e. the $D_i$'s are smooth hypersurfaces in transversal position).
\end{enumerate}

If $f$ is $G$-equivariant for some group $G$ acting on $U$ and $M$, then we require all maps to be $G$-equivariant in the definition.
\end{definition}
\begin{remark}
  Notice that if $f$ is injective then $\bar f:Y\to M$ is a resolution of the closure of $f(U)$.
\end{remark}
We can use the existence of proper normal crossing extensions to define the ($G$-equivariant) motivic Chern class:
 \begin{definition}\label{de:mc}
Suppose that  $f:U\to M$ is a map of smooth $G$-varieties and let $\bar f:Y\to M$ be a proper normal crossing extension of $f$.
 For $I\subset\underline{s}=\{1,2,\dots,s\}$ let $D_I=\bigcap_{i\in I}D_i$, $f_I=\bar f{|D_I}$, in particular $f_\emptyset=\bar f$.
Then
\begin{equation}\label{alternating-definition}\mC(f):=\sum_{I\subset \underline{s}}\;(-1)^{|I|}f_{I*}\lambda_y(T^*D_I)\,.\end{equation}
 \end{definition}
It is explained in \cite{FRW} that this is a good definition: independent of the PNC chosen.

\begin{definition} Let $N$ be a smooth variety. Then $g:N\to M$ is motivically transversal
\footnote{This idea appeared already in an early version of \cite{Ohmoto2}.}  to $f:U \to M$ if there is a PNC $\bar f:Y\to M$ for $f$ such that $g$ is transversal to all the $D_I$'s. 
\end{definition}

\begin{theorem} If $g:N\to M$ is motivically transversal to $f:U \to M$ then
\[ \mS(\tilde{f})=g^*\mS(f),\]
where $\tilde{f}: \tilde{U}\to N$ is the map in the pull-back diagram
\[\xymatrix{\tilde{U}\ar[r]^{\tilde{f}} \ar[d] &N \ar[d]^g \\ U  \ar[r]^{f}& M. &} \]
\end{theorem}
The proof is a straightforward consequence of the fact that a motivically transversal pull-back of a PNC is PNC. We can extend the notion to non-smooth varieties $Z$ requiring the existence of a stratification of $Z$  such that $g$ is motivically transversal to the restrictions of $f$ to these strata.

Being motivically transversal is a very restrictive condition. However it holds in some special situations:
\begin{proposition} \label{trans2allorbit} Let $U,M$ be smooth and a connected group $G$ is acting on $M$. Assume that $f:U\to M$ is transversal to all $G$-orbits. Then $f$ is motivically transversal to all $G$-invariant subvarieties of $M$.
\end{proposition}
The proof is simple and can be found in the proof of Lemma 5.1 in \cite{FRW}.

\subsection{$G$-equivariant motivic Segre class as a universal motivic Segre class for degeneracy loci}
Now we prove a statement (Corollary \ref{locus}) expressing the fact that the $G$-equivariant motivic Segre class is a universal formula for motivic Chern classes of degeneracy loci. The analogous statement for the Segre version of the CSM class was proved in \cite{Ohmoto2}. In K-theory the proof is simpler because reference to classifying spaces and maps can be avoided. Suppose that $\pi_P:P\to M$ is a principal $G$-bundle over the smooth $M$ and $A$ is a smooth $G$-variety. Then we can define a map
\begin{equation}\label{association}a:K_G(A)\to K(P\times_G A)  \end{equation}
by association: For any $G$-vector bundle $E$  over $A$ the associated bundle $P\times_G E$ is a vector bundle over $P\times_G A$.

\smallskip
In the rest of the paper we will use the notation $\mS(A,B)$ for $\mS(f:A\to B)$ when the map $f$ is clear from the context. The diagrams 
\[
\xymatrix{
P \ar[d]_{\pi_P}^{G-principal}& & & G\acts(i:Y\subset A) & 
P\times_G Y 
\ar[rd]_{\pi_Y}
\ar@{^{(}->}[rr]^{i_P}
& & P\times_G A \ar[ld]^{\pi_A}\\
M & & & & & M & 
}
\]
will be useful when reading the proof of the next Proposition.

\begin{proposition} Let $Y\subset A$ be $G$-invariant. Then
\[ \mS(P\times_G Y, P\times_G A)=a\big(\mS_G(Y, A)\big).\]
\end{proposition}

\begin{proof} To calculate the left hand side we need to calculate $\lambda_y(T^*P\times_G A)$ first. For the tangent space we have
\[  T(P\times_G A)=\pi_A^*(TM)\oplus (P\times_G TA),  \]
where $\pi_A:P\times_G A\to M$ is the projection. Consequently
\[ \lambda_y(T^*(P\times_G A))=\pi_A^*\lambda_y(T^*M)\cdot a(\lambda_y^G(T^*A)).\]
Assume first that $Y$ is a closed submanifold of $A$. To ease the notation we will use the $\lambda(M)=\lambda_y(T^*M)$ abbreviation. Then
 \[\mS(P\times_G Y, P\times_G A)=\frac{i_*^P\lambda(P\times_G Y)}{\lambda(P\times_G A)}=
                                       \frac{i_*\big(\pi_Y^*\lambda(M)a(\lambda^G( Y)\big)}{\pi_A^*\lambda(M)a(\lambda^G(A))},        \]
 where $\pi_Y:P\times_G Y\to M$ is the projection and $i:Y\to A$, $i_P:P\times_G Y\to P\times_G A$ are the inclusions. Then noticing that $\pi_Y=\pi_A\circ i_P$ and applying the adjunction formula we arrive at the right hand side.

 For general $Y$ we can use Definition \ref{de:mc} to reduce the calculation to the smooth case.
\end{proof}

\begin{corollary} \label{locus}Suppose that $\sigma:M\to P\times_G A$ is a section motivically transversal to $P\times_G Y$. Then
\begin{equation}\label{locuseq}
\mS(Y(\sigma), M)=\sigma^*a\big(\mS_G(Y, A)\big),
\end{equation}
where $Y(\sigma)=\sigma^{-1}(P\times_G Y)$ is the $Y$-locus of the section $\sigma$.

If $A$ is a vector space then $\sigma^*$ can be identified with the identity map  $K(P\times_GA)\simeq K(P/G)=K(M)$.
\end{corollary}

\subsection{GIT quotients} With the applications in mind we use the following simple
version of GIT quotient: Let $V$ be a $G$-vector space for a connected algebraic group $G$ and assume that $P\subset V$ is an open $G$-invariant subset such that $\pi:P\to P/G$ is a principal $G$-bundle over the smooth $M:=P/G$. (We want $P$ to be a right $G$-space so we define $pg:=g^{-1}p$.)

To state the main theorem we first introduce the \emph{K-theoretic Kirwan map} $\kappa: K_G(V)\to K(M)$ as the composition of the pull back $K(P\times_G V)\to K(M)$ via the zero section and the association map $K_G(V)\to K(P\times_G V)$ given in Equation \eqref{association} (cf. the RHS of \eqref{locuseq}). Notice again that this definition is simpler than the cohomological analogue.

\begin{theorem} \label{te:git} Let $V$ be a $G$-vector space and assume that $P\subset V$ is an open $G$-invariant subset such that $\pi:P\to P/G$ is a principal $G$-bundle over the smooth $M:=P/G$. Let $Y\subset M$. Then
\[\mS(Y, M)=\kappa(\mS_G(\pi^{-1}(Y), V)).\]
\end{theorem}

To prove this theorem we apply Corollary \ref{locus} to the \emph{universal section}: The inclusion $j:P\to V$ is obviously $G$-equivariant therefore induces a section $\sigma_j:M\to P\times_G V$.
In other words $\sigma_j(m)=[p,p]$ for any $p$ with $\pi(p)=m$.
\begin{example}\label{ex:grassmann} \rm
  For $V=\Hom(\C^k,\C^n)$ with the $G=\GL_k(\C)$-action we can choose $P:=\Sigma^0(\C^k,\C^n)$, the set of injective maps, so $P/G=\gr_k(\C^n)$. Then $P\times_G V$ is the vector bundle $\Hom(\gamma^k,\C^n)$ for the tautological subbundle $\gamma^k$ and $\sigma_j$ is the section expressing the fact that $\gamma^k$ is a subbundle of the trivial bundle of rank $n$. More generally see Remark \ref{rem:GIT}.

\end{example}

The reason we call $\sigma_j$ universal is the following simple observation: For any $Y\subset M$ the $\pi^{-1}(Y)$-locus of $\sigma_j$ is $Y$. Therefore Theorem \ref{te:git} is a consequence of the following:
\begin{proposition} The universal section $\sigma_j$ is motivically transversal to $P\times_G \pi^{-1}(Y)$ for any $Y\subset M$.
\end{proposition}

In fact we can replace $\pi^{-1}(Y)$ with any $G$-invariant (algebraic) subset  $Z\subset V$.
\begin{proof} The statement is local so it is enough to study the restriction of $\sigma_j$ to an open subset of $M$ over which $P\times_G V$ is trivial. A local trivialization can be obtained by a $\varphi:W\to P$ local section of $P$ which is transversal to the fibers i.e. to the $G$-orbits. In this local trivialization the section $\sigma_j|_W$ is the graph of the map $j\varphi:W\to V$, therefore  $\sigma_j$ is motivically transversal to $P\times_G Z$ if $\varphi$ is motivically transversal to $Z\subset V$, which is implied by Proposition \ref{trans2allorbit}.
\end{proof}

With the same argument we can obtain an equivariant version of Theorem \ref{te:git}:
\begin{theorem} \label{te:eq-git}
Let $V$ be a $G\times H$-vector space and assume that $P\subset V$ is an open $G\times H$-invariant subset such that $\pi:P\to P/G$ is a principal $G$-bundle over the smooth $M:=P/G$. Let $Y\subset M$ be $H$-invariant. Then
\[\mS_H(Y, M)=\kappa_H(\mS_{G\times H}(\pi^{-1}(Y), V)),\]
where $\kappa_H$ is the equivariant Kirwan map: the composition of the pull back $K_H(P\times_G V)\to K_H(M)$ via the zero section and the association map $K_{G\times H}(V)\to K_H(P\times_G V)$.
\end{theorem}

As a consequence we proved argument (e) in Section \ref{MCmain} (cf. Remark \ref{rem:GIT}): the calculation of motivic Chern classes of matrix Schubert cells leads directly to the calculation of motivic Chern classes of ordinary Schubert cells.
The formulas for motivic Chern classes are modified by the respective total Chern classes of the ambient spaces, while the formulas for the motivic Segre classes are identical (cf. Remark \ref{segre}).

\end{document}